\documentclass[a4paper, 12pt]{article}
\usepackage{amsmath,amsthm,amssymb}
\usepackage{color}
\usepackage{makeidx}
\usepackage{hyperref}
\makeindex 

\newcommand{\Title}{Title}

\numberwithin{equation}{section}

{\theoremstyle{definition}\newtheorem{definition}{Definition}[section]

\newtheorem{defititle}[definition]{\Title}

\newtheorem{remark}[definition]{Remark}
\newtheorem{remarks}[definition]{Remarks}
\newtheorem{example}[definition]{Example}
}
\newtheorem{proposition}[definition]{Proposition}
\newtheorem{proposition-definition}[definition]{Proposition-Definition}
\newtheorem{lemma}[definition]{Lemma}
\newtheorem{theorem}[definition]{Theorem}
\newtheorem{corollary}[definition]{Corollary}

\newcommand{\cG}{\mathcal{G}}
\newcommand{\R}{\mathbb{R}}
\newcommand{\N}{\mathbb{N}}

\newcommand{\cF}{\mathcal{F}}

\newcommand{\cL}{\mathcal{L}}

\newcommand{\cV}{\mathcal{V}}

\newcommand{\C}{\mathbb{C}}

\newcommand{\cA}{\mathcal{A}}

\newcommand{\cI}{\mathcal{I}}

\newcommand{\Z}{\mathbb{Z}}
\newcommand{\cP}{\mathcal{P}}
\newcommand{\cU}{\mathcal{U}}
\newcommand{\cH}{\mathcal{H}}
\newcommand{\cM}{\mathcal{M}}

\newcommand{\resp}{{\it resp.}\/ }
\newcommand{\id}{{\hbox{id}}}
\newcommand{\ie}{{\it i.e.}\/ }
\newcommand{\eg}{{\it e.g.}\/ }
\newcommand{\cf}{{\it cf.}\/ }

\newcommand{\dom}{{\rm dom}}
\newcommand{\gog}{\mathfrak{g}}

\newcommand{\Pseudodif}{\Psi^\infty_c}
\newcommand{\Pseudo}[1]{\Psi^{{#1}}_c}

\def\gpd{\,\lower1pt\hbox{$\longrightarrow$}\hskip-.24in\raise2pt
             \hbox{$\longrightarrow$}\,}

\begin{document}

\renewcommand\theenumi{\alph{enumi}}
\renewcommand\labelenumi{\rm {\theenumi})}

\begin{center}
{\Large\bf Pseudodifferential calculus on a singular foliation
\footnote{This research was supported in part by Funda\c{c}\~{a}o
para a Ci\^{e}ncia e a Tecnologia (FCT) through the Centro de Matem\'{a}tica
da Universidade do Porto $\langle \text{www.fc.up.pt/cmup} \rangle$}

\bigskip

{\sc by Iakovos Androulidakis \footnote{IA devotes this work to Odysseas' 0th birthday.} and Georges Skandalis}
}
\end{center}

{\footnotesize
 Centro de Matem\'{a}tica da Universidade do Porto
\vskip-4pt Edifício dos Departamentos de Matem\'{a}tica da FCUP 
\vskip-4pt Rua do Campo Alegre 687, 4169--007 Porto
\vskip-4pt e-mail: iandroulidakis@fc.up.pt

\vskip 2pt Institut de Math{\'e}matiques de Jussieu, UMR 7586
\vskip -4ptCNRS - Universit\'e Diderot - Paris 7
\vskip-4pt 175, rue du Chevaleret, F--75013 Paris
\vskip-4pt e-mail: skandal@math.jussieu.fr
}
\bigskip
\everymath={\displaystyle}

\begin{abstract}\noindent
In a previous paper (\cite{AndrSk}), we associated a holonomy
groupoid and a $C^*$-algebra to any singular foliation $(M,\cF)$. Using these,
we construct the associated pseudodifferential calculus. This calculus gives meaning to a Laplace operator of any singular foliation $\cF$ on a compact manifold $M$, and we show that it can be naturally understood as a positive, unbounded, self-adjoint operator on $L^{2}(M)$.
\end{abstract}

\section*{Introduction}

This paper is a continuation of our previous paper \cite{AndrSk}. There we defined the holonomy groupoid of any singular foliation $\cF$ on a smooth manifold $M$. Although this groupoid is a rather ill behaved object, we could define 
\begin{itemize}
\item the convolution algebra $\cA(M,\cF)$ of ``smooth compactly supported'' functions on this groupoid;
\item the full and reduced $C^*$-algebra of the foliation which are suitable (Hausdorff) completions of this convolution algebra.
\end{itemize}

A key notion in \cite{AndrSk} is that of a bi-submersion which will be also of importance here. This is loosely speaking a cover of an open subset of the holonomy groupoid. It is given by a manifold $U$ with two submersions $s,t : U \to M$, each of which lifts the leaves of $\cF$ to the fibers of $s$ and $t$. 

Here, we proceed and construct the longitudinal pseudodifferential calculus for our foliations.

\medskip The longitudinal differential operators are very easily defined: They are
generated by vector fields along the foliation.

\medskip The longitudinal pseudodifferential operators are obtained as images of distributions on bi-submersions with ``pseudodifferential singularities'' along a bisection:

  Let $(U,t,s)$ be a bi-submersion and $V\subset U$ an identity
bisection. Denote by $N$ the normal bundle to $V$ in $U$ and let
$a$ be a (classical) symbol on $N^*$. Let $\chi$ be a smooth
function on $U$ supported on a tubular neighborhood of $V$ in $U$
and let $\phi:U\to N$ be an inverse of the exponential map (defined
on the neighborhood of $V$). A pseudodifferential kernel on $U$ is
a (generalized) function $k_a:u\mapsto \int a(p(u),\xi)
\exp(i\phi(u)\xi) \chi (u)\,d\xi$ (here $p:U\to M$ is the composition $U\stackrel{\phi}{\longrightarrow} N\stackrel{q}{\longrightarrow} V$  where $q$ is the vector bundle projection $(x,\xi)\mapsto x$ - the integral is an oscillatory
integral, taken over the vector space $N^*_{p(u)}$).

\medskip The principal symbol of such an operator is a homogeneous function on a locally compact space $\cF^*$ which is a family of vector spaces (of non constant dimension).

\bigskip As in the case of foliations and Lie groupoids (\cf \cite{Connes0,
 Monthubert-Pierrot, Nistor-Weinstein-Xu, Vassout}), we show:

\begin{itemize} \item The kernel $k_a$ defines a multiplier of $\cA(M,\cF)$ (more precisely, of the image of $\cA(M,\cF)$ in the $C^*$-algebra of the foliation).

\item Those multipliers  form an algebra. 

\item The algebra of pseudodifferential operators is filtered by
the order of $a$. The class of $k_a$ only depends up to lower
order on the germ of the principal part of $a$ on  $\cF^*$.

\item Negative order pseudodifferential operators are elements of
the $C^*$-algebra (full and therefore reduced) of the foliation.

\item  Zero order pseudodifferential operators define bounded multipliers of
the $C^*$-algebra of the foliation. 

\item We therefore have an exact sequence of $C^*$-algebras $$0\to C^*(M,\cF)\to
\Psi^*(M,\cF)\to B \to 0$$ where $\Psi^*(M,\cF)$ denotes
the closure of the algebra of zero order pseudodifferential operators and $B$ is (a quotient of) the algebra $C_0(S^*\cF)$ of continuous functions on the ``cosphere bundle'' which vanish at infinity.


\item Longitudinally elliptic operators of positive order (\ie operators whose principal symbol is invertible when restricted to $\cF^*$) give rise to regular quasi-invertible operators.

\item We may form a Laplacian of $\cF$, which is an example of such a positive order longitudinally elliptic  operator. It defines a regular positive self-adjoint multiplier of the $C^*$-algebra, and therefore a positive self-adjoint operator in any non-degenerate representation of $C^{*}(\cF)$; in particular a self-adjoint element of $B(L^{2}(M))$.
\end{itemize}

\bigskip One can also take coefficients on a smooth vector bundle over $M$. This allows to build an index theory, which we intend to treat in a subsequent paper.

\bigskip The paper is organized as follows:
\begin{itemize}
\item In section 1 we recall basic facts about pseudodifferential calculus: we define distributions on a manifold $U$ with singularities on a submanifold $V$ and state the main classical results that will be used in the subsequent sections, namely: 
\begin{enumerate}
\item Such distributions have a principal symbol which is a smooth function in the co-shpere bundle of $N^*$, where $N$ is the normal bundle of $V$ in $U$.
\item We discuss pull backs and push forwards (partial integrations) of pseudodifferential distributions.
\item If $V_{1}, V_{2}$ are transversal to each other then the product of $P_{1} \in \cP(U,V_{1})$ and $P_{2} \in \cP(U,V_{2})$ is a well defined distribution; a partial integral gives rise to an element of $\cP(W,V_{1} \cap V_{2})$ whose principal symbol is the product of the principal symbols.
\item The algebra $C^{\infty}_{c}(U)$ is dense in $\cP(U,V)$.
\end{enumerate}

\item In section 2, for the convenience of the reader, we briefly recall the framework we introduced in \cite{AndrSk} and give some slight modifications of results there. We moreover define the ``cotangent space'' $\cF^{*}$ together with its natural locally compact topology. 

\item In section 3 we define the longitudinal pseudodifferential operators. Namely, we define an algebra $\Psi^{\infty}(\cU,\cV)$ of pseudodifferential operators associated with an atlas of bi-submersions $\cU$ and a family of identity bisections $\cV$ covering $M$. 

For this algebra to be defined reasonably, one needs to bear in mind the following: In case the foliation is regular (or defined by a Lie groupoid), the longitudinal pseudodifferential operators form a a subalgebra of the multipliers of the groupoid convolution algebra. A general singular foliation may not arise from a Lie groupoid, but it always comes from an atlas of bi-submersions. So, in order for $\Psi^{\infty}(\cU,\cV)$ to generalize properly the pseudodifferential calculus of the regular case, we define its elements a priori as multipliers of the image of the convolution algebra $\cA_{\cU}$ in its Hausdorff completion $C^{*}(\cU)$.

This is achieved by showing in \S 3.3 that every pseudodifferential kernel $P \in \cP(U,V)$ defines a multiplier $\tilde{\theta}_{U,V}(P)$ of the image of the natural morphism $\theta : \cA_{\cU} \to C^{*}(\cU)$ from $\cA_{\cU}$ to its Hausdorff completion. To this end, we need to show in \S 3.2 that a non-degenerate representation $\Pi$ of $C^{*}(\cU)$ on a Hilbert space $\cH$ admits an appropriate extension to compactly supported pseudodifferential kernels.

In \S 3.5 we show that our pseudodifferential operators have a principal symbol which is a homogeneous function on the subset of non-zero elements in $\cF^{*}$. Last, in \S 3.6, we show that thus defined, pseudodifferential operators form a $*$-algebra.

\item In section 4 we show that our longitudinal pseudodifferential calculus has the classical ellipticity properties. Namely the existence of parametrices for elliptic operators and the existence of square roots for even order, self-adjoint operators with positive principal symbol.

\item In section 5 we establish the extension $$0\to C^*(M,\cF)\to
\Psi^*(M,\cF)\to B \to 0$$ discussed above.

\item In section 6 we show how our pseudodifferential calculus allows for the Laplacian of a singular foliation to be realized as a self-adjoint element of $B(L^{2}(M))$.

\end{itemize}

\tableofcontents

\section{Generalized functions with pseudodifferential singularities}

In this section we recall some well known facts on pseudodifferential distributions and operators.

\subsection{Symbols}

The symbols that we consider are the ``classical'' or ``polyhomogeneous'' symbols. Let us briefly recall how they are defined:

\begin{itemize}
\item Let $k,n\in \N$, $V$ be an open subset of $\R^n$. For $m\in \Z$, define the space $S^{m}(V \times \R^k)$ of symbols of order (less than or equal to) $m$ to be the set of smooth functions $a : V \times \R^k\to \C$ such that for any compact set $K \subset V$ and any multi-indices $\alpha\in \N^n$ and $\beta\in \N^k$ there is a constant $C_{K,\alpha, \beta}\in \R_+$ such that, for all $x\in K$ and $\xi\in \R^n$ we have  $$|\partial^{\alpha}_{x}\partial^{\beta}_{\xi}a(x,\xi)| \leq C_{K,\alpha, \beta}(1 + |\xi|)^{m - |\beta|}.$$ 

\item A symbol $a \in S^{m}(V \times \R^k)$ is called {\em classical} or \emph{polyhomogeneous} if $a \sim \sum_{k = -\infty}^{m}a_{k}$, where $a_k$ are {\em positively homogeneous} functions of degree $k$ in the second variable, namely they satisfy $a_{k}(x,t\xi) = t^{k}a(x,\xi)$ for all $\xi \neq 0$ and $t > 0$. The notation ``$\sim$" means that $a(x,\xi) - \chi(\xi)\sum_{k = m-M+1}^{m}a_{k}(x,\xi) \in S^{m-M}(V \times \R^k)$ for all $M \in \N$. Here $\chi$ is a cut-off function with $\chi(\xi) = 0$ if $|\xi| < 1/2$ and $\chi(\xi) = 1$ if $|\xi|\geq 1$. Note that this property does not depend on the cut-off function $\chi$. We will consider only classical symbols in this paper.

\item These notions are diffeomorphism invariant, and thus allow to define symbols on vector bundles:  given a smooth manifold $V$ and a smooth vector bundle $N$ over $V$, we may define the space of classical symbols $S^{m}_{cl}(V,N)$ on the bundle $N$: these are functions $a$ on the total space $N$ which admit an expansion $a\sim \sum_{k=-\infty}^m a_k$  as above in any chart where the bundle $N$ is trivial. 

\item We will be mostly interested in the subspace  $S^{m}_{cl,c}(V,N)\subset S^{m}_{cl}(V,N)$ of symbols whose support is compact on the $V$ direction, \ie such that there exists a compact subset $K$ in $V$ with $a(\xi)=0$ whenever $p(\xi)\not\in K$ ($p:N\to V$ is the bundle projection).

\item The above definitions extend to give spaces $S^{m}_{cl}(V,N;E)$ and $S^{m}_{cl,c}(V,N;E)$ of symbols with values in a smooth vector bundle $E$ over $V$ (considering $E$ as a subbundle of a trivial bundle). 
\end{itemize}

\subsection{Pseudodifferential generalized functions and submersions}

\begin{remark}[on densities] 
In order to make our constructions (which use integration) independent on choices of Lebesgue measures, we use densities everywhere. We just indicate which densities one has to take, with no further explanations most of the time.
\end{remark}

\subsubsection{Distributions transverse to a submersion}

Let $M,N$ be manifolds, $p:N\to M$ a submersion and $E$ a vector bundle on $M$. Let $P\in C_c^{-\infty}(N;\Omega^1\ker dp\otimes p^*E)$ be a distribution with compact support on $N$. It defines a distribution $p_!P\in C_c^{-\infty}(M;E)$ by a formula $\langle p_!P,f\rangle=\langle P,f\circ p\rangle$ ($f\in C^\infty(M;\Omega^1M\otimes E^*$). 

Let $F$ be a vector bundle on $N$. A distribution $P\in C^{-\infty}(N;F)$ on $N$ is said to be \emph{transverse} to $p$ if for every $f\in C_c^\infty (N;\Omega^1\ker dp\otimes F^*)$, the distribution $p_!(f.P)$ is smooth on $M$. If $P$ is transverse to $p$, it restricts to a distribution $P'$ on $N'=p^{-1}(M')$ for every submanifold $M'$ of $M$. The distribution $P'$ is obviously transverse to the restriction $p':N'\to M'$ of $p$.

\subsubsection{Generalized functions with pseudodifferential singularities}

\paragraph{On a vector bundle.}
Let $V$ be a smooth manifold and $N$ a smooth vector bundle over $V$. A symbol $a\in S^m_{cl,c}(V,N^*;\Omega^{1}N^*)$ defines a generalized function of pseudodifferential type on the total space of $N$ which is given by a formal expression (for $u\in N$):
$$P_a(u)=\int_{N^*_{p(u)}} a(p(u),\xi) e^{i\langle u,\xi\rangle}$$
where $p:N\to V$ is the bundle map and the integral is an  ``oscilatory integral''. This ``function'' $P_a$ makes  sense as a distribution on the total space of $N$ \ie elements of  the dual space of the space of smooth function compact support on the total space of $N$ (actually smooth sections of a suitable bundle of one densities). Furthermore, the image of this distribution along the map $p:N\to V$ is (defined and) smooth; we may therefore consider $P_a$ as a $C^\infty (V)$ linear map from $C_c^\infty (N;\Omega^{1}N)\to C^\infty (V)$ through a formula (for $v\in V$ - here $k$ is the dimension of the bundle $N$)
\begin{eqnarray} 
\label{eq11}
\langle P_a,f\rangle(x)=(2\pi)^{-k}\int_{N^*_x\times N_x} a(x,\xi) e^{-i\langle u,\xi\rangle}f(u)=(2\pi)^{-k}\int_{N^*_x} a(x,\xi) \hat f(\xi).
\end{eqnarray}
Integrating along the manifold $V$ we obtain a distribution on the total space of $N$. We will sometimes write (formally) \begin{eqnarray} 
\label{eq12}
P_a=(2\pi)^{-k}\int_{N^*_x} a(x,\xi) e^{-i\langle u,\xi\rangle}
\end{eqnarray}
Almost by definition, the distribution $P_a$ is transverse to the projection $p:N\to M$.

\paragraph{Along a submanifold.}
Let $U$ be a smooth manifold and $V$ a closed smooth submanifold of $U$. Denote by $N$ the normal bundle to $V$.

We will use the tubular neighborhood construction. Let us briefly fix the notation: this is given by a neighborhood $U_1$ of $V$ in $U$ and a local diffeomorphism  $\phi:U_1\to N$  such that, for $v\in V$,  $\phi(v)=(v,0)$ and $d\phi $ restricted to $V$ is the identity in the normal direction. More explicitly, note that for $v\in V$,  $T_{(v,0)}N=T_vV\oplus N_v$; the above condition means that $d\phi_v$ composed with the second projection is the projection $T_vU\to N_v=T_vU/T_vV$.

A \emph{generalized function} on $U$ \emph{with pseudodifferential singularity} on $V$ is a generalized function, which far from $V$ is smooth, and near $V$ coincides with a generalized function of pseudodifferential type through a tubular neighborhood construction.  

In other words $P$ is of the form $P=h+\chi \cdot P_a \circ \phi$ where
\begin{itemize}
\item $(U_1,\phi )$ is a tubular neighborhood construction as above,
\item $h\in C^\infty(U)$,
\item $\chi $ is a smooth ``bump'' function equal to $1$ in a a neighborhood of $V$ and to $0$ outside $U_1$;
\item $a\in S^m_{cl}(V,N^*;\Omega^{1}N^*)$ is a (classical) symbol.
\end{itemize}

Concretely, such a pseudodifferential function is a distribution on $U$: if $f\in C_c^\infty(U;\Omega^1(TU))$, we put
$$\langle P,f\rangle=\int_U h(u)f(u)+(2\pi)^{-k}\int _{N^*U_1}a(p\circ \phi (u),\xi)\chi(u)f(u)e^{-i\langle \phi(u),\xi\rangle}$$

The generalized functions on $U$ with  pseudodifferential singularities on $V$ form a vector space that will be denoted by $\cP(U,V)$. We denote by  $\cP_c(U,V)$ those which vanish outside a compact subset of $U$ (\ie of the form $\chi P$ where $\chi\in C_c^\infty(U)$ and $P\in \cP(U,V)$).

\begin{example}\label{vectorfields}
 Choose a metric on $U$ and thus a trivialization of all densities. Let $X$ be a vector field with compact support on $U$. The map $q_X:f\mapsto \int_V Xf$ is an example of a (pseudo)differential distribution. Note that if $X$ is tangent to $V$, then $\int_V Xf=-\int_V {\rm div}(X)f$. In other words, $q_X$ only depends up to order zero operators on the image of $X$ in the normal bundle. 
\end{example}

Let us state a few facts about these generalized functions that we will use extensively:

\begin{itemize}
\item We may extend the construction  of pseudodifferential functions and define pseudodifferential \emph{sections} of any smooth (complex) vector bundle $E$ over $U$. These also give rise to distributions as above, \ie linear mappings on $C_c^\infty(U;\Omega^{1}TU\otimes E^*)$. We denote by $\cP(U,V;E)$ the space they form - and $\cP_c(U,V;E)$ the subspace of those with compact support.

\item The space of generalized functions with pseudodifferential singularities doesn't depend on the choice of $\phi:U\to N$ with the above requirements. 
\end{itemize}

We immediately deduce:

\begin{proposition}
A pseudodifferential distribution $P\in \cP(U,V;E)$ is transverse to any submersion $p:U\to M$ which is transverse to $V$. \hfill $\square$
\end{proposition}

Notice that the smooth function on $U\setminus V$ associated with a generalized function $P$ doesn't determine $P$: if the symbol is a polynomial - then $P$ is differential and is supported by $V$ - \ie vanishes outside $V$.

\subsubsection{Density of smooth functions}\label{density}

Let $P\in \cP(U,V)$ be given by a formal formula $$P(u)=h(u)+(2\pi)^{-k}\chi(u)\int _{N^*_{p(u)}}a(p\circ \phi (u),\xi)e^{-i\langle \phi(u),\xi\rangle}$$ Let $\chi_1$ be a smooth nonnegative function with compact support on $\R_+$ which is equal to $1$ in a neighborhood of $0$. Put then $$P_n(u)=h(u)+(2\pi)^{-k}\chi(u)\int _{N^*_{p(u)}}a(p\circ \phi (u),\xi)\chi_1(\|\xi\|/n)e^{-i\langle \phi(u),\xi\rangle}$$ Then $P_n\in C_c^\infty(U)$ and converges to $P$ in the topology of $C^{-\infty}$. Furthermore, for every submersion $q:U\to M'$ which is transverse to $V$, and every $f\in C^\infty(U;\Omega^1\ker dp)$, the sequence of $p_!(fP_n)$ of smooth functions on $M'$  converges to $p_!(fP)$ in the topology of $C_c^\infty(M')$.

\subsubsection{Principal symbol}

A generalized function $P\in \cP(U,V)$ of order $m$ with pseudodifferential singularities has a principal symbol. If $P$ is associated with a symbol $a$ of order $m$, then the principal symbol $\sigma_{m}(P)$ of $P$ is the homogeneous part of $a$ of order $m$. It is defined outside the zero section on the total space of $N$ and $\sigma_{m}(P)(x,\xi)$ is a $1$-density on $N_x^*$ (for $x\in V$ and $\xi\in N_x^*$ non zero). By choosing smoothly a euclidean metric of the bundle $N$, it can be defined as an element $\sigma_{m}(P) \in C^{\infty}(S^{*}N)$, where $S^{*}N$ is the co-sphere bundle of $N^{*}$. 

\begin{proposition} \label{exactsequence}
We have an exact sequence $$0 \to \Psi^{m-1}(V,U) \to \Psi^{m}(V,U) \stackrel{\sigma_{m}}{\longrightarrow} C^{\infty}(S^{*}N) \to 0$$
\end{proposition} 
\begin{proof}
The only thing that has to be proved is that $\sigma_m(P)$ only depends on $P$. This is a classical fact (see \eg \cite{Nistor-Weinstein-Xu}). Let us recall this briefly:

 We may assume $U=N$. Then $P$ defines a $C^\infty(V)$-linear map $C^\infty_c(N)\to C^\infty_c(V)$ (using appropriate densities). Let $x\in V$ and $\xi\in N_x$ a non zero covector. Then $\sigma_m(x,\xi)=\lim_{\tau\to +\infty} (i\tau)^{-m}P(e^{i\tau\varphi}\chi )(x)$ where $\varphi\in C_c^\infty(N)$ with  derivative $\xi$ at $x\in V$ (the zero point of $N_x$) along $N_x$ - and $\chi \in C_c^\infty(N)$ is equal to $1$ in a neighborhood of $x$.
\end{proof}
We may of course add bundles into the picture: if $P\in \cP(U,V;E)$ then $\sigma_{m}(P)(x,\xi)\in \Omega^1(N_x^*)\otimes E_x$ (for $x\in V$ and $\xi\in N_x^*$ non zero). 

It is easy to see that generalized functions satisfy all the properties of classical pseudodifferential operators. For future reference in this sequel we recall the following one; it is the key ingredient that provides the existence of parametrices for elliptic pseudodifferential operators.

\begin{theorem}\label{smoothing}
Let $(Q_n)_{n\in \N}$ be a sequence of pseudodifferential functions such that $Q_n-Q_{n+1}$ is of order $m-n$. Then there exists a pseudodifferential function $Q$ such that $Q-Q_n$ is of order $m-n$ for all $n$.
\end{theorem}

\begin{example}
The principal symbon of $q_{X}$ in example \ref{vectorfields}  is $\xi\mapsto i\langle X|\xi\rangle$.
\end{example}

\subsection{Pull-back, push-forward, product}

\subsubsection{Pull-back (restriction)}

Let $U$ and $U'$ be smooth manifolds and $V\subset U$ a closed submanifold.  Let  $p:U'\to U$ be a smooth map, transverse to $V$ and put $V' = p^{-1}(V)$. It is a submanifold of $U'$. Let $P\in \cP(U,V)$. Locally (near a point of $V'$) we may assume $U=V\times \R^k,\ \ U'=V'\times \R^k$ and $p(x',u)=(p(x'),u)$ for $x'\in V'$ and $u\in \R^k$. 

We may then assume that $P$ is given (formally - see equation \ref{eq12})) by a formula $$P(x,u)=(2\pi)^{-k}\int _{\R^k} e^{-i\langle u,\xi\rangle} a(x,\xi)+h(x,u)$$ where $h $ is smooth and $a$ is a symbol.

We then define $p^*P$ setting $(p^*P)(x',u)=P(p(x'),u)$
 
 Under the identification of the normal bundle $N'$ of $V'$ in $U'$ with $p^*N$, the principal symbol of  $p^*P$ is given by  $\sigma_{m}(p^*P)=\sigma_{m}(P)\circ p$.

\subsubsection{Push-forward (partial integration)}

\begin{proposition}[\bf Push-forward] \label{push-forward}
Let $U$ and $U'$ be smooth manifolds and $V\subset U$ a closed submanifold.  Let $p:U\to U'$ be a submersion which restricts to a diffeomorphism $p:V\to V'$ where $V'$ is a submanifold of $U'$. Let $E'$ be a vector bundle on $U'$. 
\begin{enumerate}
\item Integration along the fibers of $p$ gives rise to a map $p_!:\cP_c(U,V;\Omega^{1}\ker dp\otimes p^*E')\to \cP _c(U',V';E')$ defined by $\langle p_!(P),f\rangle=\langle P,f\circ p\rangle$  for $f\in C^\infty (U;\Omega^1TU\otimes E^*)$. The principal symbol of $p_!P$ is $\sigma'_m(x,\xi)=\sigma_m(x,p^*\xi)$ where $x\in V'\simeq V$, and $p^*$ is the (injective) map $(T_xU'/T_xV)^*\to (T_xU/T_xV)^*$ induced by $p$.
\item If $p$ is onto, then $p_!$ is onto too.
\end{enumerate}
\end{proposition} 

\begin{proof}
\begin{enumerate}
\item We may assume that $p : V \times \R^{k}\times \R^{\ell}  \to V \times \R^k$ is the projection and that $P$ is given by a formula $$\langle P,f\rangle =(2\pi)^{-(k+\ell)}\int e^{-i(\langle u,\xi\rangle +\langle v,\eta\rangle)}a(x,\xi,\eta)\chi_1(u)\chi_2(v)f(x,u,v).$$
We thus get $$\langle p_{!}P,f \rangle =\langle P,f \circ p \rangle = (2\pi)^{-(k+\ell)} \int_{V\times \R^{\ell }\times (\R^{k+\ell })^{*} }a(x,\xi,\eta)e^{-i\langle u,\xi\rangle}\chi_1(u)\widehat{\chi_2}(\eta) f(x,u)$$ Using a Taylor expansion of the form $$a(x,\xi,\eta)\sim a(x,\xi,0)+\sum_{1\le |\alpha|}\frac{\eta^{\alpha}}{\alpha !}\frac{\partial^{|\alpha|}}{(\partial \eta)^\alpha} a(x,\xi,0)$$ we find the principal term $$ (2\pi)^{-(k+\ell)} \int_{V\times \R^{\ell }\times (\R^{k+\ell })^{*} }a(x,\xi,0)e^{-i\langle u,\xi\rangle}\chi_1(u)\widehat{\chi_2}(\eta) f(x,u).$$ Since $(2\pi)^{-\ell}\int \widehat{\chi_2}(\eta) =\chi_2(0)=1$ we find $$(2\pi)^{-k} \int_{V\times \R^{\ell }\times (\R^{k})^{*} }a(x,\xi,0)e^{-i\langle u,\xi\rangle}\chi_1(u) f(x,u)  $$

\item Let $P'\in \cP _c(U',V';E')$. One obviously may extend the principal symbol of $P'$ to get a homogeneous section on the normal bundle of $V$ in $U$. It follows that there exists an operator $P_1\in \cP_c (U,V;p^*E')$ such that $p_!P_1-P'$ is of order $m-1$. Using induction, one constructs a sequence $P_n\in \cP_c (U,V;p^*E')$ such that $p_!P_n-P'$ is of order $m-n$ and $P_{n+1}-P_n$ is of order $m-n$. Using theorem \ref{smoothing} , one then gets $Q$ such that $Q-P_n$ is of order $m-n$, whence $p_!Q-P'$ is smoothing. Finally, using partitions of the identity, it is obvious that $p_!:C_c^\infty(U;p^*E')\to C_c(U',E')$ is onto.\qedhere
\end{enumerate}
\end{proof}

\begin{remarks}\label{push-forward-remarks} \begin{enumerate}
\item We will also need a slightly more general statement:\\
In the above proposition, we may just assume that $p$ induces a submersion $p:V\to V'$ where $V'$ is a submanifold of $U'$. In that case, for $P\in \cP(U,V;\Omega^1(\ker dp)\otimes p^*(E'))$, the principal symbol $\sigma'(x',\xi')$ of $p_!P$ is the integral of $\sigma(x,p_x^*(\xi'))$ for $x$ running in the fiber $V\cap p^{-1}(x')$ (and $p_x^*:(T_{x'}U'/T_{x'}V')^*\to (T_xU/T_xV)^*$ is the (injective) map induced by $(dp)_x$) (\footnote{An easy check shows that the densities match correctly.}).

To establish this, one may assume $U'=V'\times \R^k$, $U=V'\times \R^j\times \R^k\times \R^\ell$, $V=V'\times \R^j\times \{(0,0)\}$ and $p$ is the obvious projection $V'\times \R^j\times \R^k\times \R^\ell\to V'\times \R^k$.

\item Obviously, we have $(q\circ p)_!=q_!\circ p_!$ if $p:U\to U'$ and $q:U'\to U''$ are submersions satisfying requirements of (a).
\end{enumerate}
\end{remarks}

\subsubsection{Products}

In order to understand the product of pseudodifferential operators in our context, we give the following Lemma. We say that a submersion $p:U\to U'$ is \emph{strictly transverse} to a submanifold $V\subset U$ if at any point $x\in V$ we have $T_xV\oplus \ker dp_x=T_xU$.

\begin{lemma}\label{linearization}
Let $U$ be a manifold, $V_1,V_2$ two closed submanifolds of $U$ that are transverse to each other and $p : U \to U'$ a submersion strictly transverse to both $V_{1}$ and $V_{2}$. Then there are charts of $U$ of the form the form $W \times \R^{k} \times \R^{k}$ covering $V = V_{1} \cap V_{2}$, such that $V_{1} = W\times \R^{k} \times \{0\}$ and $V_{2} = W \times \{0\} \times \R^{k}$  and $p$ can be written as $(v,\xi,\eta) \mapsto (v,\xi + \eta)$.
\end{lemma}

\begin{proof}
Notice that $V$ is a manifold due to the transversality of $V_{1}$ to $V_{2}$. Then $V$ is covered by charts such that $U$ can be written as $U' \times T$, where $T = \R^{k}$ is the fiber of $p$, and $p$ is the projection. Since $T$ is transverse to $V_{1}$, for this chart we can write $V_{1} = U' \times \{0\}$. Since $V_{2}$ is trasversal to $V_{1}$, it is the graph of a submersion $q : V_{1} \to T$; we have $V=q^{-1}(\{0\})$, hence we can write $V_{1} = V \times \R^{k}$ and $q$ the projection. Under these identifications, we found a chart $V\times \R^k\times \R^k$, for which $V_1=V\times \R^k\times \{0\}$, $V_2=\{(x,\xi,\eta);\ \xi=\eta\}$ and $p(x,\xi,\eta)=\xi$. The result follows by applying the diffoemorphism  $(v,\xi,\eta) \mapsto (v,\xi + \eta, \eta)$ on $V \times \R^{k} \times \R^{k}$.
\end{proof}

\begin{proposition} \label{products}
Let $U$ be a manifold and $V_1,V_2$ two closed submanifolds of $U$ that are transverse to each other. Let $P_i\in \cP(U,V_i)$. 
\begin{enumerate}
\item The product $P_1\cdot P_2$ makes sense as a distribution on $U$.

\item Assume that $P_1\cdot P_2$ has compact support. Let $p:U\to U'$ be a submersion which is both strictly transverse to $V_1$ and $V_2$ and whose restriction to $V_1\cap V_2$ is injective and proper. Then $p_!(P_1.P_2)$ is a pseudodifferential function. If then $\sigma_{i}$ is the principal symbol of $P_{i}$ ($i = 1,2$), then the principal symbol of $p_!(P_1.P_2)$ is $p_!(\sigma_1) \cdot p_{!}(\sigma_2)$.
\end{enumerate}
\end{proposition} 

\begin{proof}
\begin{enumerate}
\item The statement is local. We may therefore assume $U=V\times \R^k\times \R^\ell$, $V_1=V\times \R^k\times\{0\}$ and $V_2=V\times \{0\}\times \R^\ell$. Also, by an obvious choice of the tubular neighborhood construction, we may write $$P_1(x,u,v)=\int e^{i\langle v|\eta\rangle} a_{1}(x,u,\eta)d\eta \ \ \ \hbox{and}\ \ \ P_2(x,u,v)=\int e^{i\langle u|\xi\rangle}a_{2}(x,\xi,v)d\xi.$$ Here, $a_{1}$ and $a_{2}$ are classical (polyhomogeneous) symbols.

The product is then given by a formula:  $$\langle P_1\cdot P_2,f\rangle =\int\left(\int e^{i\langle u|\xi\rangle}e^{i\langle v|\eta\rangle}a_{1}(x,u,\eta)a_{2}(x,\xi,v)f(x,u,v)\,dudv\right)dxd\xi d\eta $$ which makes perfect sense when $f\in C_c^\infty (U)$.

\item Due to \ref{linearization} we can assume $U = V \times \R^{k} \times \R^{k}$, $U' = V \times \R^{k}$ and $p(x,u,v) = (x,u+v)$. 

We thus have to compute $$\bigg(\int\left(\int e^{i\langle u|\xi\rangle}e^{i\langle v|\eta\rangle}\chi (x,u,v) a_{1}(x,u,\eta)a_{2}(x,\xi,v)f(x,u+v)\,dudv\right)d\xi d\eta\bigg)dx .$$
This is a ``classical'' oscillatory integral on $\R^{4k}$ which is treated by the usual integration by parts methods. It actually amounts to composing (families indexed by $V$) of pseudodifferential operators in $\R^{k}$. 
\qedhere
\end{enumerate}
\end{proof}

\section{Singular foliations; cotangent space}

\subsection{Foliations, bi-submersions, atlas, *-algebra}

 Let us first recall some definitions, notation and  conventions taken in \cite{AndrSk}.

\subsubsection{Foliations}

\begin{definition}\begin{enumerate}
\renewcommand{\theenumi}{\alph{enumi}}
\renewcommand{\labelenumi}{\theenumi)}

\item Let $M$ be a smooth manifold. A \emph{foliation} on $M$ is a
locally finitely generated submodule of $C_c^\infty(M;TM)$ stable
under Lie brackets.

\item There is an obvious notion of a pull-back foliation: if
$(M,\cF)$ is a foliation and $f:L\times M\to M$ is the second
projection, the pull back foliation $f^{-1}(\cF)$ the space of
vector fields whose $M$ component as a map from $L$ to
$C_c^\infty(M,TM)$ takes its values in $\cF$. In the same way, one
defines pull back foliations by submersions. See \cite{AndrSk},
subsection 1.2.3. 

\item For $x\in M$, put $I_{x} = \{ f \in C^{\infty}(M) : f(x)=0 \}$. The \textit{fiber} of $\cF$ is the quotient $\cF_{x} = \cF/I_{x}\cF$. The \textit{tangent space of the leaf} is the image $F_{x}$ of the evaluation map $ev_{x} : \cF \to T_{x}M$. 

The spaces $F_{x}$ and $\cF_{x}$ differ on singular leaves: the dimension of $F_{x}$ is lower semi-continuous and the dimension of $\cF_{x}$ is upper semi-continuous. They coincide in a dense open subset of $M$, namely on points $x$ that lie in a regular leaf.

We get a surjective linear map $e_{x} : \cF_{x} \to F_{x}$ whose kernel is a Lie algebra $\gog_{x}$.
\end{enumerate}
\end{definition}

\subsubsection{Bi-submersions, bisections}
The main ingredient in the constructions of \cite{AndrSk} is the notion of bi-submersion that we now recall.

\begin{definition}
A \emph{bi-submersion} of $(M,\cF)$ is a
smooth manifold $U$ endowed with two smooth maps $s,t:U\to M$  which are submersions and
satisfy:\begin{enumerate}\renewcommand{\theenumii}{\roman{enumii}}
\renewcommand{\labelenumii}{(\theenumii)}
\item $s^{-1}(\cF)=t^{-1}(\cF)$.
\item $s^{-1}(\cF)=C_c^\infty(U;\ker ds)+C_c^\infty (U;\ker dt)$.
\end{enumerate}
If $(U,t_{U},s_{U})$ and $(V,t_{V},s_{V})$ are bi-submersions then $(U,s_{U},t_{U})$ is a bi-submersion, as well as $(W,s_{W},t_{W})$ where $W = U \times_{s_{U},t_{V}}V$, $s_{W}(u,v) = s_{V}(v)$ and $t_{W}(u,v) = t_{U}(u)$ (\cite[Prop. 2.4]{AndrSk}). 

The bi-submersion $(U,s_{U},t_{U})$ is called the inverse of $(U,t_U,s_U)$ and is noted $(U,t_U,s_U)^{-1}$, or just $U^{-1}$; the bi-submersion $(W,s_{W},t_{W})$ is called the composition of $(U,t_{U},s_{U})$ and $(V,t_{V},s_{V})$ and is noted $(U,t_{U},s_{U})\circ (V,t_{V},s_{V})$ - or just $U\circ V$.
\end{definition} 

In \cite[Prop. 2.10]{AndrSk} it is shown that there are enough bi-submersions $\{(U_{i},t_{i},s_{i})\}_{i \in I}$ such that $\bigcup_{i \in I} s_{i}(U_{i}) = M$: For an $x \in M$, if $X_{1},\ldots,X_{n} \in \cF$ form a base of $\cF_{x}$ then we can find a neighborhood $U$ of $(x,0)$ in $M \times \R^{n}$ where the exponential $t_{U}(y,\xi) = \exp(\sum_{i = 1}^{n}\xi_{i}X_{i})(y)$ is defined and such that $(U,t_{U},s_{U})$ is a bi-submersion, where $s_{U}$ denotes the first projection. Such a bi-submersion is sometimes called an \textit{identity bi-submersion}.

\begin{definition}[morphisms of bi-submersions]  Let $(U_i,t_i,s_i)$ ($i=1,2$) be bi-submersions. A smooth map $f:U_1\to U_2$ is a \emph{morphism of bi-submersions} if $s_1=s_2\circ f$ and $t_1=t_2\circ f$.
\end{definition}

In order to compare bi-submersions, we used the notion of bisections:

\begin{definition}
A \textit{bisection} of $(U,t,s)$ is a locally closed submanifold $V$ of $U$ such that $s$ and $t$ restricted to $V$ are diffeomorphisms to open subsets of $M$.  We say that $V$ is an \emph{identity bisection} if moreover the restrictions of $s$ and $t$ to $V$ coincide. 

We say that $u \in U$ \textit{carries} the foliation-preserving local diffeomorphism $\phi$ if there exists a bisection $V$ such that $u\in V$ and $\phi = t\mid_{V} \circ (s\mid_{V})^{-1}$.
\end{definition}

In \cite[\S 2.3]{AndrSk} it is shown that if $(U_{j},t_{j},s_{j})$ are bi-submersions, $j = 1,2$ then a $u_{1} \in U_{1}$ carries the same diffeomorphism with $u_{2} \in U_{2}$ iff there exists a morphism of bi-submersions $g$ defined in an open neighborhood of $u_{1} \in U_{1}$ such that $g(u_{1}) = u_{2}$.

Actually the proof given in \cite[\S 2.3]{AndrSk} proves a stronger statement which will be useful here:

\begin{proposition} \label{identbisect}
Let $(U_{j},t_{j},s_{j})$ be bi-submersions, $j = 1,2$, $V_i\subset U_i$ identity bisections and $u_{j} \in V_{j}$ such that $s_1(u_1)=s_2(u_2)$. Then there exists a morphism of bi-submersions $g$ defined in an open neighborhood $U'_1$ of $u_{1} \in U_{1}$ such that $g(u_{1}) = u_{2}$ and $g(V_1\cap U'_1)\subset V_2$. \hfill$\square$
\end{proposition}

\subsubsection{Minimal bi-submersions}
\begin{definition}
 If $(U,t,s)$ is a bi-submersion and $u\in U$, then the dimension of the manifold $U$ is at least $\dim M+\dim \cF_{s(u)}$. We say that $(U,t,s)$ is \emph{minimal} at $u$ if $\dim U=\dim M+\dim \cF_{s(u)}$. 

If $f:(U',t',s')\to (U,t,s)$ is a morphism of bi-submersions and $U$ is minimal at $f(u')$, then $df_{u'}$ is onto. Therefore, there is a neighborhood $W'$ of $u'$ in $U'$ such that the restriction of $f$ to $W'$ is a submersion.

For every bi-submersion $(U,t,s)$ and every $u\in U$, there exists a bi-submersion $(U',t',s')$, and $u'\in U'$ such that $U'$ is minimal at $u'$ and carries at $u'$ the same diffeomorphisms as $U$ at $u$. It follows that there is a neighborhood $W$ of $u$ in $U$ and a submersion which is a morphism $f:(W,t_{|W},s_{|W})\to (U',t',s')$.
\end{definition}

The following result will be used in the sequel:

\begin{proposition}\label{transverse}
Let $(U,t,s)$ be a bi-submersion and $V$ an identity bisection. Let $u \in V$ and assume $U$ is minimal at $u$. Then there is a neighborhood $U'$ of $u$ in $U$ and a submersion $p : U' \circ U' \to U$ which is a morphism of bi-submersions which is strictly transverse to $U \circ V$ and to $V \circ U$.
\end{proposition}

\begin{proof}
The composition $U\circ U$ carries at $(u,u)$ the identity bisection $V\circ V$. It follows that there exists a neighborhood of $(u,u)$ in $U\circ U$ thet we may assume of the form $U'\circ U'$ and a morphism $p:U'\circ U'\to U$. By minimality of $U$ at $u$, we may assume that $p$ is a submersion. Moreover $U\circ V$ and $V\circ U$ are bi-submersions. Again, by minimality of $U$ at $u$, it follows that, up to restricting $U'$, the restrictions of $p$ to $(U\circ V)\cap (U'\circ U')$ and $(V\circ U)\cap (U'\circ U')$ are submersions, hence $p$ is transverse to  $U \circ V$ and to $V \circ U$ - strictly by equality of dimensions.
\end{proof}

\subsubsection{Atlas of bi-submersions}

\begin{definition}
 Let $\cU = \big((U_{i},t_{i},s_{i})\big)_{i \in I}$ be a family of bi-submersions. A bi-submersion $(U,t,s)$ is \emph{adapted} to $\cU$ if for all $u \in U$ there exists an open subset $U' \subset U$ containing $u$, an $i \in I$, and a morphism of bi-submersions $U' \to U_{i}$. 

We say that $\cU$ is an \textit{atlas} if
\begin{enumerate}
\item $\bigcup_{i \in I}s_i(U_{i}) = M$.
\item The inverse of every element in $\cU$ is adapted to $\cU$.
\item The composition $U \circ V$ of any two elements in $\cU$ is adapted to $\cU$.
\end{enumerate}
An atlas $\cV = \{(V_{j},t_{j},s_{j})\}_{j \in J}$ is adapted to $\cU$ if every element of $\cV$ is adapted to $\cU$. We say $\cU$ and $\cV$ are \textit{equivalent} if they are adapted to each other. There is a \textit{minimal atlas} which is adapted to any other atlas: this is the atlas generated by ``identity bi-submersions''. 
\end{definition}

\subsubsection{The groupoid of an atlas}

The groupoid of an atlas $\cU = \big((U_{i},t_{i},s_{i})\big)_{i \in I}$ is the quotient of $U=\coprod_{i\in I}U_i$ by the equivalence relation for which $u\in U_i$ is equivalent to $u'\in U_j$ if $U_i$ carries at $u$ the same local diffeomorphisms as $U_j$ at $u'$.

For every bi-submersion $U$ adapted to $\cU$ we have a well defined (quotient) map $\zeta_U :U\to \cG_\cU$.

\subsubsection{The C*-algebra of a foliation}

In \cite[\S 4.3]{AndrSk} we associated to an atlas $\cU$ a $*$-algebra $\cA(\cU) = \bigoplus_{i \in I}C^\infty_{c}(U_{i};\Omega^{1/2}U_{i})/\cI$. Here $\Omega^{1/2}$ denotes the bundle of half densities on $\ker ds\oplus \ker dt$ and $\cI$ is the space spanned by the $p_{!}(f)$, where $p : W \to U$ is a submersion which is a morphism of bi-submersions and $f \in C_{c}(W;\Omega^{1/2}W)$ is such that there exists a morphism $q : W \to V$ of bi-submersions which is a submersion and such that $q_{!}(f) = 0$. 

To any bi-submersion $V$ adapted to $\cU$ we can associate a linear map $Q_{V} : C^\infty_{c}(V;\Omega^{1/2}V) \to \cA_{\cU}$. Involution and convolution in $\cA_{\cU}$ are then defined by $$\Big(Q_{V_{1}}(f_{1})\Big)^*= (Q_{V_{1}^{-1}})(f_{1}^*) \ \ \hbox{and}\ \
Q_{V_{1}}(f_{1})Q_{V_{2}}(f_{2})=Q_{V_{1}\circ V_{2}}(f_{1}\otimes f_{2})$$ for $(V_{i},t_{i},s_{i})$ bi-submersions adapted to $\cU$ and $f_{i} \in C_{c}^{\infty}(V_{i};\Omega^{1/2}V_{i})$, $i = 1,2$.

\medskip The $C^*$-algebra $C^{*}(\cU)$ of the atlas $\cU$ is the (Hausdorff)-completion of $\cA_{\cU}$ with a natural $C^*$-norm \cite[\S4.4, \S4.5]{AndrSk}. Actually, two natural completions were considered: the full and the reduced one. 

If $\cU$ is adapted to $\cV$, we have natural $*$-morphisms $\cA(\cU)\to \cA(\cV)$ and $C^*(\cU)\to C^*(\cV)$.

When $\cU$ is the mininal atlas we write $C^{*}(M,\cF)$ for the full and $C^{*}_{r}(M,\cF)$ for the reduced completion.

The representations of the full $C^*$-algebras were described in \cite[\S5]{AndrSk} in terms of representations of the associated groupoid. We will come back to this description below (\S\ref{repsPseudo}).

\subsection {The cotangent space}

The symbols should be functions on the total space of a ``cotangent bundle''. Let us discuss this space.

\begin{definition}
The {\em cotangent bundle} of the foliation $\cF$ (although it is in general not a bundle) is the union  $\cF=\coprod_{x\in M} \cF^*_x$. It is endowed with a natural projection $p:\cF^*\to M$ ($(x,\xi)\mapsto x$). Also, for each $X\in \cF$, we have a natural map $q_X:(x,\xi)\mapsto \xi\circ e_x(X)$. We endow $\cF^*$ with the weakest topology for which the maps $p$ and $q_X$ are continuous.
\end{definition}

\begin{proposition}\label{fiberdual}
The space $\cF^{*}$ is locally compact. 
\end{proposition}
\begin{proof} It is enough to show that $p^{-1}(U)$ is locally compact for every small enough open set $U$ of $M$. We may therefore assume that $\cF$ is finitely generated  \ie it is a quotient of $C_c^\infty(M;\R^n)$. Then $\cF^*$ consists of elements of $(x,y)\in M\times (\R^n)^*$ such that the map $C_c^\infty(M;\R^n)\ni \varphi\mapsto \langle y,\varphi(x)\rangle$ factors through the quotient $\cF$ of $C_c^\infty(M;\R^n)$ (\ie vanishes on the kernel of the map $C_c^\infty(M;\R^n)\to \cF$). It is a closed subset of $M\times (\R^n)^*$ and is therefore locally compact.
\end{proof}

\begin{example} \label{ex:SO3}
Let $\cF$ be the foliation on $\R^{3}$ defined by the (infinitesimal) action of $SO(3)$. It is easy to see that $\cF^{*} = \cup_{\xi \in \R^{3}}\{ x \in \R^{3} : \langle x|\xi \rangle = 0 \}$.§
\end{example}

Let $(U,t,s)$ be a bi-submersion and $V$ an identity bisection. Identifying the normal bundle $NV$ with $\ker ds$ (or $\ker dt$), there are epimorphisms $d_{x}t : N_{x}V \to \cF_{x}$, $x \in V$ (or $d_{x}s$). Dualizing these maps we find that locally the cotangent bundle $\cF^{*}$ is a closed subspace of $N^{*}V$. Thus we can restrict symbols to $\cF^{*}$. 

Let $(U',t',s')$ and $(U,t,s)$ be bi-submersions with identity bisections $V'$ and $V$. Let $p : U' \to U$ is a smooth map that is a morphism of bi-submersions such that $p(V')\subseteq V$. Then for every $x \in s'(V')$ the inclusion $\cF_{x}^{*} \to (N'_{x})^{*}$ factors as the composition of $p_{x}^{*} : N_{x}^{*} \to (N'_{x})^{*}$ with the inclusion $\cF_{x}^{*} \to N_{x}^{*}$.

\section{The space of pseudodifferential kernels on a foliation}

\subsection{Pseudodifferential kernels on a bi-submersion}

Let $(U,t,s)$ be a bi-submersion. The bundle (over $U$) of half densities on $\ker ds\oplus \ker dt$  will be simply denoted by $\Omega^{1/2}$.

We define $\cP^m(U,V;\Omega^{1/2})$ (\resp $\cP_c^m(U,V;\Omega^{1/2})$) to be the space of generalized sections of the bundle $\Omega^{1/2}$ with pseudodifferential singularities along $V$  (\resp those with compact support) of order $\le m$. We drop the superscript $m$ when we make no order requirements.

Let $N$ be the normal bundle of the inclusion $V\to U$. Note that $N$ canonically identifies with both the restrictions to $V$ of the bundles $\ker ds$ and $\ker dt$. Under these identifications the bundle $\Omega^{1}N^*\otimes \Omega^{1/2}(\ker ds)\otimes \Omega^{1/2}(\ker dt)$ is trivial.

The principal symbol of $P\in \cP^m(U,V;\Omega^{1/2})$  is therefore an element of $S^{m}(V,N^*)$, \ie a function on $N^*$ which is homogeneous of degree $m$.

\begin{definition}\label{defi3.1}
Let $(U,t,s)$ be a bi-submersion and $V$ an identity bisection. 
\begin{enumerate}
\item Denote by $U^{-1}$ the inverse bi-submersion and $\kappa : U \to U^{-1}$ the (identity) isomorphism. For an operator $P \in \cP_c(U,V;\Omega^{1/2})$ define $P^{*} = \overline{\kappa_{!}P}$.

\item Finally, let $U,U'$ be bi-submersions, $V\subset U$ an identity bisection, $P\in \cP_c(U,V;\Omega^{1/2})$ and $f\in C_c^\infty(U';\Omega^{1/2})$. Using the first projection which is a submersion $U\circ U'\to U$ we may pull back $P$ to a generalized function with pseudodifferential singularities on $V\circ U'$; multiplying it by $f$, we get $P\star f\in \cP_c(U\circ U',V\circ U';\Omega^{1/2})$. In the same way, we construct $f\star P\in \cP_c(U'\circ U,U'\circ V;\Omega^{1/2})$. 
\end{enumerate}
\end{definition}

\subsection{Extending representations to pseudodifferential kernels}\label{repsPseudo}

We fix an atlas $\cU$. Denote by $\cG_\cU$ the associated groupoid. 

\subsubsection{Quasi-invariant measures}

Recall that a measure $\mu$ on $M$ is said to be quasi-invariant by the atlas $\cU$ if for every bi-submersion $(U,t,s)$ adapted to the atlas $\cU$ the measures $\mu\circ \lambda^s$ and $\mu \circ \lambda^t$ are equivalent (here $\lambda^s$ and $\lambda^t$ are Lebesgue measures along the fibers of $s$ and $t$ respectively). 

In that case, there is a measurable almost everywhere invertible section $\rho^U$ of $\Omega^{-1/2}(\ker ds)\otimes \Omega^{1/2}(\ker dt)$ such that for every $f\in C_c(U;\Omega^{1/2}(\ker ds)\otimes \Omega^{1/2}(\ker dt))$ we have $$\int_M\Big(\int_{s^{-1}(x)} (\rho^U_u)^{-1}\cdot f(u)\Big)d\mu (x)=\int_M\Big(\int_{t^{-1}(x)} \rho^U_u\cdot  f(u)\Big)d\mu (x).$$

If $p:(U,t,s)\to (U',t',s')$ is a morphism of bi-submersions, there is a canonical isomorphism of the bundles $\Omega^{-1/2}(\ker ds)\otimes \Omega^{1/2}(\ker dt)$ and $p^*\Big(\Omega^{-1/2}(\ker ds')\otimes \Omega^{1/2}(\ker dt')\Big)$. Under this isomorphism, the maps $\rho^U$ and $p^*(\rho^{U'})$ coincide (almost everywhere). We interpret this by saying that $\rho$ is defined on the groupoid $\cG_\cU$. Furthermore, one may see that it is naturally a homomorphism.

\begin{remark}
The morphism $\rho $ defined here takes also into account the Radon Nykodym derivative $(D^U)^{1/2}$ that we used in \cite[\S 5]{AndrSk}.
\end{remark}

\begin{remark}[on measurable functions on $\cG_\cU$]
A function $f$ on $\cG_\cU$ is just a function $f\circ \zeta$ on $\coprod_{i\in I} U_i$ which is constant on the equivalence classes. We will say that the function $f$ is measurable with respect to the quasi-invariant measure $\mu$ if $f\circ \zeta_U$ is measurable (with respect to the measures $\mu\circ \lambda$).
\end{remark}

\begin{lemma} \label{lemma3.4}
Let $(W,t,s)$ and $(W_i,t_i,s_i)$ be bi-submersions, $Y\subset W$ a submanifold and $p^i:W\to W_i$ be morphisms of bi-submersions which are submersions transverse to $Y$ (here $i\in \{1,2\}$). Let $\mu $ be a quasi-invariant measure on $M$ and $\beta $ a measurable bounded function on $\cG_\cU$. 

For $Q\in \cP_c(W,Y;\Omega^{1/2})$, we have $$\int_M\bigg(\int_{t_1^{-1}\{x\}}\rho^{W_1}_w\cdot p^1_!(Q)(w)\beta\circ \zeta_{W_1}(w)\bigg)d\mu(x)=\int_M\bigg(\int_{t_2^{-1}\{x\}}\rho^{W_2}_w\cdot p^2_!(Q)(w)\beta\circ \zeta_{W_2}(w)\bigg)d\mu(x)$$
\end{lemma}
Note that, by the transversality assumption, the $p^i_!(Q)$ are smooth functions. This explains the meaning of this formula. 

\begin{proof}
For $Q\in C_c^\infty (W;\Omega^{1/2})$ these two expressions coincide (by Fubini) with $$\int_M\bigg(\int_{t^{-1}\{x\}}\rho^{W}_u\cdot Q(u)\beta\circ\zeta_U(u)\bigg)d\mu(x).$$
The Lemma follows from \S \ref{density}.
\end{proof}

We will use an immediate generalization of this lemma:

\begin{lemma} \label{lemma3.5}
Let $(W,t,s)$ and $(W_i,t_i,s_i)_{i\in I}$ be bi-submersions and $Y\subset W$ a closed submanifold. Assume that there is an open cover $(Z_i)_{i\in I}$ of $W$ and $p^i:Z_i\to W_i$  morphisms of bi-submersions which are submersions transverse to $Y\cap Z_i$. Let $(\chi_i)$ be a smooth partition of the identity adapted to the cover $Z_i$. Let $\mu $ be a quasi-invariant measure on $M$ and $\beta $ a measurable bounded function on $\cG_\cU$. 

For $Q\in \cP_c(W,Y;\Omega^{1/2})$, the quantity
 $$\int_M\sum_{i\in I}\bigg(\int_{t_i^{-1}\{x\}}\rho^{W_1}_w\cdot p^i_!(\chi_iQ)(w)\beta\circ \zeta_{W_i}(w)\bigg)d\mu(x))$$ does not depend on the choices of $I$ and  the family $(Z_i,\chi_i,W_i,p_i)_{i\in I}$ with the above requirements. \hfill$\square$
\end{lemma}

Let $(W,t,s)$ be a bi-submersion and $Y\subset W$ a submanifold. We say that $Y$ is \emph{transverse to $\zeta_W$} if there is an open cover $(Z_i)_{i\in I}$ of $W$ and $p^i:Z_i\to W_i$  morphisms of bi-submersions which are submersions transverse to $Y\cap Z_i$. The above Lemma makes sense of $$\int_M\bigg(\int_{t^{-1}\{x\}}\rho^{W}_u\cdot Q(u)\beta\circ\zeta_U(u)\bigg)d\mu(x)$$ for $Q\in \cP_c(W,Y;\Omega^{1/2})$.

\subsubsection{Extension of representations}

We now fix an atlas $\cU$ and a non degenerate $*$-representation  $\Pi$ of $C^*(\cU)$ on a Hilbert space $\cH$. 

Let us fix some notation: 

Denote by $\theta:\cA_\cU\to C^*(\cU)$ the natural morphism from $\cA_\cU$ to its Hausdorff-completion. If $U$ is a bi-submersion adapted to $\cU$, denote by $\theta_U:C_c^\infty(U;\Omega^{1/2})\to C^*(\cU)$ the composition $C_c^\infty(U;\Omega^{1/2})\stackrel{Q_U}{\longrightarrow} \cA_\cU\stackrel{\theta}{\longrightarrow}  C^*(\cU)$. Finally, put $\Pi_U=\Pi\circ \theta_U$.

According to \cite[\S 5]{AndrSk}, there is a triple $(\mu, H, \pi)$ where:
\begin{enumerate}\renewcommand{\theenumi}{\alph{enumi}}
\renewcommand{\labelenumi}{\theenumi)}
\item $\mu$ is a quasi-invariant measure on $M$;
\item $H =(H_{x})_{x \in M}$ is a measurable (with respect to $\mu$) field
of Hilbert spaces over $M$.
\item For every bi-submersion  $(U,t,s)$ adapted to $\cU$,  $\pi^U$ is a measurable (with respect
to $\mu\circ \lambda $) section of the field of unitaries
$\pi^U_u:H_{s(u)}\to H_{s(u)}$.
\end{enumerate}

Moreover:
\begin{enumerate}
\item $\pi $ is `defined on $\cG_\cU$':\\
if $f:U\to V$ is a morphism of bi-submersions, for almost all
$u\in U$ we have $\pi_u^U=\pi_{f(u)}^V$.
\item $\pi $ is a homomorphism: \\
If $U$ and $V$ are bi-submersions adapted to $\cU$, we have
$\pi^{U\circ V}_{(u,v)}=\pi^U_{u}\pi^V_{v}$ for almost all
$(u,v)\in U\circ V$.
\end{enumerate}

Then $\cH=\int^\oplus_MH_x\, d\mu(x)$ is the space of square integrable sections of $H$. For every bi-submersion $(U,t,s)$ adapted to $\cU$, we have:
$$\Pi_U(f)(\xi)(x) = \int_{t^{-1}\{x\}} (\rho^U_u \cdot f(u))\pi^U_{u}(\xi_{s(u)})$$
$\mu$-a.e. for every $ \xi \in
\cH$ and $x \in M$.

It follows that $\langle \eta,\Pi_U(f)\xi\rangle =\int _M\Big(\int_{t^{-1}\{x\}} (\rho^U_u \cdot f(u))\langle \eta_x,\pi^U_u(\xi)\rangle \Big)d\mu(x)$.

\begin{proposition}
Let $(W,t,s)$ be a bi-submersion and $Y\subset W$ a submanifold transverse to $\zeta_W$. There is a linear map $\Pi_{W,Y}:\cP_c(W,Y;\Omega^{1/2})\to \cL(H)$ such that for every open subset $Z\subset W$, every morphism $p:Z\to U$ of bi-submersions which is a submersion transverse to $Y\cap Z$ and every $Q\in \cP_c(Z,Y\cap Z;\Omega^{1/2})\subset \cP_c(W,Y;\Omega^{1/2})$ we have $\Pi_{W,Y}(Q)=\Pi_U(p_!(Q))$.
\end{proposition}

\begin{proof}
Let $\xi ,\eta\in \int^\oplus H_x\, d\mu(x)=\cH$ be bounded square integrable sections, and define a bounded measurable function $\beta$ on $\cG_\cU$ by putting $\beta\circ \zeta_U(u)=\langle \eta_x,\pi^U_u(\xi_{s_{U}(u)})\rangle$.  It follows from Lemma \ref{lemma3.5} that, for $Q\in \cP_c(W,Y;\Omega^{1/2})$, using any partition of unit adapted to a nice cover of $W$ to construct $\Pi_{W,Y}(Q)$ we have $$\langle \eta,\Pi_{W,Y}(Q)\xi\rangle=\int_M\bigg(\int_{t^{-1}\{x\}}\rho^{W}_u\cdot Q(u)\beta\circ\zeta_U(u)\bigg)d\mu(x)$$ 
which is well defined by Lemma \ref{lemma3.5}. The conclusion follows by density of bounded sections in $\cH$.
\end{proof}

Taking $\Pi$ to be faithful, we find:

\begin{corollary} \label{cor3.7}
Let $(W,t,s)$ be a bi-submersion and $Y\subset W$ a submanifold transverse to $\zeta_W$. There is a linear map $\theta_{W,Y}:\cP_c(W,Y;\Omega^{1/2})\to \theta(\cA(\cU))$ such that for every open subset $Z\subset W$, every morphism $p:Z\to U$ of bi-submersions which is a submersion transverse to $Y\cap Z$ and every $Q\in \cP_c(Z,Y\cap Z;\Omega^{1/2})\subset \cP_c(W,Y;\Omega^{1/2})$ we have $\theta_{W,Y}(Q)=\theta_U(p_!(Q))$. \hfill $\square$
\end{corollary}

\subsection{Pseudodifferential kernels and multipliers}

Let $(U,t,s)$ be a bi-submersion and $V\subset U$ an identity bisection. Let $U'$ and $U''$ be bi-submersions. For $P\in \cP_c(U,V;\Omega^{1/2})$, $f\in C_c^\infty(U';\Omega^{1/2})$ and $g\in C_c^\infty(U'';\Omega^{1/2})$, we defined $P\star f\in \cP_c(U\circ U',V\circ U';\Omega^{1/2})$ and $g\star P\in \cP_c(U''\circ U,U''\circ V;\Omega^{1/2})$ (see Definition \ref{defi3.1}).

Note that $V\circ U'$ and $U''\circ V$ are bi-submersions and therefore transverse to $\zeta_{U\circ U'}$ and $\zeta_{U''\circ U}$ respectively.

\begin{theorem}
Let $(U,t,s)$ be a bi-submersion and $V\subset U$ an identity bisection. Let $U'$ and $U''$ be bi-submersions. For $P\in \cP_c(U,V;\Omega^{1/2})$, $f\in C_c^\infty(U';\Omega^{1/2})$ and $g\in C_c^\infty(U'';\Omega^{1/2})$ we have $\theta_{U''}(g)\theta_{U\circ U',V\circ U'}(P\star f)=\theta_{U''\circ U,U''\circ V}(g\star P)\theta_{U'}(f)$. In other words, there is a multiplier $\tilde \theta_{U,V}(P)$ of $\theta(\cA(\cU))$ such that $$\tilde \theta_{U,V}(P)\theta_{U'}(f)=\theta_{U\circ U',V\circ U'}(P\star f)\ \ \hbox{and}\ \ \theta_{U''}(g)\tilde \theta_{U,V}(P)=\theta_{U''\circ U,U''\circ V}(g\star P).$$
\end{theorem}

\begin{proof}
It is enough to prove this theorem for $P,f,g$ with small enough support so that we may assume that there exist morphisms of bi-submersions $U\circ U'\to W'$ and $U''\circ U\to W''$ which are submersions respectively transverse to $V\circ U'$ and $U''\circ V$. Then the morphisms $\id\times p:(u'',u',u')\mapsto (u'',p(u,u'))$ and $q\times \id$ are morphisms of bi-submersions and are submersions respectively transverse to $U''\circ V\circ U'$.

We have $\theta_{U''}(g)\theta_{U\circ U',V\circ U'}(P\star f)=\theta_{U''\circ W'}((\id\times p)(g\star P\star f))=\theta_{U''\circ U\circ U',U''\circ V\circ U'}(g\star P\star f)$  by Corollary \ref{cor3.7}. In the same way $\theta_{U''\circ U,U''\circ V}(g\star P)\theta_{U'}(f)=\theta_{U''\circ U\circ U',U''\circ V\circ U'}(g\star P\star f).$
\end{proof}

Note that $\tilde \theta_{U,V}(P)$ is a closable multiplier, since its adjoint contains $\tilde \theta_{U^{-1},V}(P^*)$ and is therefore densely defined.

Let $p:U'\to U$ be a submersion and a morphism of bi-submersions such that $p(V')\subset V$. Recall from prop. \ref{push-forward} that there is a natural map $p_!:\cP(U',V';\Omega^{1/2})\to \cP(U,V;\Omega^{1/2})$ obtained by integration along fibers. 

\begin{proposition} \label{imagepseudo}
With the above notation, we have $\tilde \theta_{U,V}\circ p_!=\tilde \theta_{U',V'}$. \hfill $\square$
\end{proposition}

\subsection{The space of pseudodifferential multipliers}

We have fixed an atlas $\cU =(U_i,t_i,s_i)_{i\in I}$ together with identity bisections $V_i\subset U_i$ such that $\bigcup_{i\in I}s_i(V_i)=M$ (\footnote{For a given $i$, $V_i$ may be empty.}). 

Denote by
$\widetilde U$ the disjoint union $\coprod_{i\in I} U_i$ and by $\widetilde V\subset \widetilde U$ the disjoint union $\coprod_{i\in I} V_i$.

 \begin{definition}
Let $m\in \Z$. We form a space $\Pseudo m(\cU,\cV)$.  This is the image  in the multiplier algebra of $\theta(\cA(\cU))$ of $\bigoplus_{i\in I} \cP^m_c(U_i,V_i;\Omega^{1/2})=\cP_c^m(\widetilde U,\widetilde V;\Omega^{1/2})$.

We define the \emph{space of pseudodifferential multipliers} to be the union  $\Pseudodif(\cU,\cV)$ of $\Pseudo m(\cU,\cV)$.

An element in $\bigcap_{m\in \Z}\Pseudo m(\cU,\cV)$ is called \emph{regularizing}. 
\end{definition}

Let $(U,t,s)$ be a bi-submersion adapted to $\cU$ and $V$ an identity bisection in $U$. We just constructed a linear map $\tilde \theta _{U,V}: \cP_c(U, V;\Omega^{1/2})\to \Pseudodif(\cU,\cV)$.

Elements of $\cA(\cU)$ give obviously rise to regularizing operators. On the other hand, a regularizing operator is for every $k$ the image of a function $f_k\in C_c^k(\widetilde U;\Omega^{1/2})$; as the map $\theta $ is not injective, it is not clear whether $f_k$ can be taken constant (\ie in $\cA(\cU)$).

\subsection{The longitudinal principal symbol}

\begin{definition}
The \emph{longitudinal principal symbol} of an element of $\cP_c^m(\widetilde U,\widetilde V;\Omega^{1/2})$ is a homogeneous function on the space $\cF^*\setminus M$ of non zero elements in $\cF^*$.

To construct it, we may assume that $s_i:V_i\to M$ are injective. Identify $V_i$ with its image in $M$. Let $P_i\in \cP(U_i,V_i;\Omega^{1/2})$ of order $m$. The principal symbol $\tilde\sigma_m(P_i)$ is defined to be the restriction to the subspace $\cF^*\setminus M$ of $N^*\setminus V$ of the principal symbol of $P_i$. Extending it by linearity, we define the longitudinal principal symbol of any element of $\bigoplus_{i\in I} \cP_c^m(U_i,V_i;\Omega^{1/2})$.
\end{definition}

\begin{remark} \label{remarksymbol}
It is not obvious whether the longitudinal principal symbol is defined in the image $\Pseudo m(\cU,\cV)$. This happens if the groupoid $\cG_\cU$ is longitudinally smooth. In that case, one may define $\tilde\sigma_m(P)(x,\xi)$ using the regular representation on $L^2((\cG_\cU)_x)$ and a formula as in Proposition \ref{exactsequence}: $\tilde\sigma_m(x,\xi)=\lim_{n\to +\infty}\langle e^{in\varphi}\chi_n, (in)^{-m}P(e^{in\varphi}\chi)(x)$ where $\varphi\in C_c^\infty((\cG_\cU)_x)$ with  derivative $\xi$ at $x$ and $\chi_n$ is a suitable function which has $L^2$ norm $1$ and has support around $x$ - it is of the form $\chi_n(y)=n^{k/2}\chi (n\|x-y\|^2)$ where $k=\dim (\cG_\cU)_x=\dim \cF_x$.

Even when this longitudinal principal symbol is well defined though, it is not clear whether there is an exact sequence as in Proposition \ref{exactsequence}, since an element in $P\in \cP_c^m(U,V;\Omega^{1/2})$ whose longitudinal principal symbol vanishes on $\cF^*\subset N^*$ may not be in $\cP_c^{m-1}(U,V;\Omega^{1/2})$. Namely, it is not clear to us whether there exists $Q\in \cP_c^{m-1}(U,V;\Omega^{1/2})$ which has the same image in $\Pseudodif (\cU,\cV)$ as $P$. Here is an example of this situation: Consider the foliation defined by the the action of $SO(3)$ in $\R^{3}$ (example \ref{ex:SO3}) and take the order $0$ symbol $a(x,\xi) = e^{-\frac{1}{\langle x|\xi \rangle^2}}$ outside $\cF^{*}$ and zero in $\cF^{*}$. This is a symbol of order $0$ which vanishes on $\cF^*$, but there is no pseudodifferential operator of order $-1$ whose symbol is $a$ in a neighborhood of $\cF^*$.
\end{remark}

\begin{remark}
If we change atlases, the pseudodifferential operators don't really change. Indeed, let $(U,t,s)$ be a bi-submersion and $V$ an identity bisection. Let $\cU$ be any atlas (\eg the minimal one). Then, since $\cU$ carries the identity bisection, there is a neighborhood $U'$ of $V$ in $U$ such that $(U',t,s)$ is adapted to $\cU$. Thus, there is an open cover $(U'_i)$ of $U'$ and morphisms $f_i:U'_i\to \widetilde U$ such that $f_i(V\cap U'_i)\subset \widetilde V$ (\cf prop. \ref{identbisect}). Every element of $\cP_c(U,V;\Omega^{1/2})$ can be written as $P+h$ with $P\in \cP_c(U',V;\Omega^{1/2})$ and $h\in C_c^\infty(U,\Omega^{1/2})$.

We deduce:
\begin{enumerate}
\item If we change the identity bisections we don't change at all the space $\Pseudodif (\cU,\cV)$. 

\item Let $\cU$ and $\cU'$ be atlases such that $ \cU$ is adapted to $\cU'$. We have a natural morphism $j:C^*(\cU)\to C^*(\cU')$. If $j$ is injective, then we have an equality $\cP(\cU',\cV')=\cP(\cU,\cV)+\theta(\cA(\cU'))$.
\end{enumerate}
\end{remark}

\subsection{Convolution: the algebra of pseudodifferential kernels}

\begin{lemma} \label{lemma3.15}
Let $U,W$ be bi-submersions adapted to $\cU$, $V\subset U$ an identity bisection and $p:U\circ U\to W$ a morphism of bi-submersions which is a submersion strictly transverse to $V\circ U$ and to $U\circ V$. Then, for $Q_1,Q_2\in \cP_c(U,V;\Omega^{1/2})$ we have $\tilde \theta(Q_1)\tilde \theta(Q_2)=\tilde \theta(p_!(Q_1\star Q_2))$. 
\end{lemma}

Note that by proposition \ref{products}, $Q_1\star Q_2$ makes sense as a distribution and $p_!(Q_1\star Q_2)$ is pseudodifferntial.

\begin{proof}
 Let $U'$ be another bi-submersion and $f\in C_c^\infty(U';\Omega^{1/2})$. We have to show that $\tilde \theta(Q_1)\tilde \theta(Q_2)\theta(f)=\tilde \theta(p_!(Q_1\star Q_2))\theta(f)$. To that end, we may take a faithful representation $\Pi$ and compute $\langle \eta,\Pi\Big(\tilde \theta(Q_1)\tilde \theta(Q_2)\theta(f)\Big)\xi\rangle $ and $\langle \eta,\Pi\Big(\tilde \theta(p_!(Q_1\star Q_2))\theta(f)\Big)\xi\rangle $. These two expressions are equal when $Q_1$ and $Q_2$ are smooth functions. The general case follows using \S \ref{density}.
\end{proof}

\begin{theorem}\label{multiplicativity}
The space $\Pseudodif (\cU,\cV)$ is a subalgebra of the multiplier algebra of $\theta (\cA(\cU))$. More precisely, given $P_i\in \cP_c^{m_i}(\widetilde U,\widetilde V;\Omega^{1/2})$ ($i=1,2$, there is $P\in \cP_c^{m}(\widetilde U,\widetilde V;\Omega^{1/2})$ such that $\tilde \theta(P_1)\tilde \theta(P_2)=\tilde \theta(P)$ and $\tilde \sigma_{m_1}(P_1)\tilde \sigma_{m_2}(P_2)=\tilde \sigma_{m}(P)$.
\end{theorem}

\begin{proof}
By prop. \ref{transverse}, there is a cover of $M$ by (open) sets $s(V'_i)$ where $V'_i$ is an identity bisection of a bi-submersion $U'_i$ adapted to $\cU$ for which there is a morphism of bi-submersions $p_i':U'_i\circ U'_i\to W_i$ which is a submersion strictly transverse to $V'_i\circ U'_i$ and to $U'_i\circ V'_i$.

There is a finite open cover $(U''_j)$ of the $Supp(P_1)\cup Supp(P_2)$ such that, putting $V''_j=U''_j\cap \widetilde V$, if $s(V''_j)\cap s(V''_k)\not=\emptyset$, then there  are morphisms of bi-submersions from $U''_j$ and from $U''_k$ to the same $U'_i$ which are submersions.

Using a partition of the identity adapted to $U''_j$, we are reduced to the case where $P_1\in \cP_c^{m_1}(U''_j,V''_j;\Omega^{1/2})$ and $P_2\in \cP_c^{m_2}(U''_k,V''_k;\Omega^{1/2})$.

 If $V''_j\circ V''_k=\emptyset$ then $P_1\star P_2\in \cP_c(U''_j\circ U''_k,W;\Omega^{1/2})$ with $W=V''_j\circ U''_k\cup U''_j\circ V''_k$, and therefore $\tilde \theta(P_1)\tilde \theta(P_2)\in \theta(\cA_\cU)$. 

If $s(V''_j)\cap s(V''_k)\not=\emptyset$, we may replace $P_1$ and $P_2$ by their images $Q_1,Q_2$ in $\cP_c(U'_i,V'_i;\Omega^{1/2})$ (prop. \ref{imagepseudo}). Now, by proposition \ref{products}, $Q_1\star Q_2$ makes sense as a distribution and $(p'_i)_!(Q_1\star Q_2)$ is pseudodifferntial and has the right longitudinal principal symbol. The result follows by Lemma \ref{lemma3.15}.
\end{proof}

\section{Longitudinal ellipticity}

In this section we assume that the manifold $M$ is compact.

\begin{definition}
A pseudodifferential operator $P \in \cP_c^m(\cU,\cV)$ is said to be \emph{longitudinally elliptic} if its longitudinal principal symbol $\tilde \sigma_m(P)$ is invertible.
\end{definition}

\subsection{Parametrix}

We now state the analogues in our setting of some most important classical results in the pseudodifferental calculus.

\begin{theorem}[Existence of quasi-inverses]\label{parametrix}
Let $P\in \cP_c^m(\widetilde U,\widetilde V;\Omega^{1/2})$ be a longitudinally elliptic operator of order $m$. There is a pseudodifferential operator $Q\in \cP_c^{-m}(\widetilde U,\widetilde V;\Omega^{1/2})$ of order $-m$ such that $1-\tilde \theta(P)\tilde \theta(Q)$ and $1-\tilde \theta(Q)\tilde \theta(P)$ are regularizing.
\end{theorem}

The main ingredient of the proof is:

\begin{lemma}\label{parametrix:main}
Let $P\in \cP_c^m(\widetilde U,\widetilde V;\Omega^{1/2})$ be longitudinally elliptic and $S\in \Pseudo k(\widetilde U,\widetilde V)$. Then there exists $Q\in \cP_c^{k-m}(\widetilde U,\widetilde V;\Omega^{1/2})$ such that $\tilde \theta(P)\tilde \theta(Q) - S\in \Pseudo {k-1} (\widetilde U,\widetilde V)$.
\end{lemma}

which, in turn, relies on the following result:

\begin{lemma}\label{localization}
Let $P\in \cP_c^m(\widetilde U,\widetilde V;\Omega^{1/2})$ be longitudinally elliptic.  Then there exists the following data: 
\begin{enumerate}\renewcommand{\theenumi}{\roman{enumi}}
\renewcommand{\labelenumi}{{\rm (\theenumi)}}
\item \label{itemaa}a finite set $I$, bi-submersions $(U'_{i},t'_{i},s'_i)_{i \in I}$ and identity bisections $V'_i\subset U_i'$ and morphisms of bi-submersions $p^i:U'_i\circ U'_i\to W_i$ that are submersions strictly transverse to $U'_i\circ V'_i$ and to $V'_i\circ U'_i$;
\item \label{itembb} open relatively compact subsets $U_i\subset U'_i$ such that $V_i=V'_{i} \cap U_{i}$ is relatively compact in $V'_i$;
\item $\bigcup_{i \in I}s_i(V_i) = M$;
\item operators $P'_{i} \in \cP_c^m(U'_{i},V'_{i};\Omega^{1/2})$ whose (plain) principal symbol $\sigma_m(P'_i)$ is invertible on $\overline{V_i}$;
\item smooth functions $\phi_{i} \in C^{\infty}_{c}(U'_{i})$ such that $\phi_{i}|_{V_{i}} = 1$
\end{enumerate}
so that $\phi_{i}P-P'_{i}$ is regularizing.
\end{lemma}

\begin{proof}[Proof of \ref{localization}]
Let $x \in M$; consider a pair $(U',V')$ such that $(U',t',s')$ is a bi-submersion, $V'$ is an identity bisection, which is minimal at a point $v\in V'$ with $s(v)=x$ . Take $\phi\in C_c^\infty(M)$ to be $1$ in a neighborhood of $x$ with support in $s(V')$. There exists an operator $P'\in \cP(U',V')$ of order $m$, such that  $\phi P - P'$ is regularizing. Whence $\sigma_{m}(P)\phi = \sigma_{m}(P')|_{\cF^{*}}$ and $\sigma_{m}(P')(v,\xi)$ is invertible for every $\xi\in N^*_{v}$ since $(U',V')$ is minimal at $v$.  The set of $w\in V'$ for which $\sigma_{m}(P')(w,\xi)$ is invertible for all $\xi\in N^*_ {w}$ is open (by compactness of the spheres). It follows that $\sigma_{m}(P')(w,\xi)$ is invertible for every $w$ in a small enough neighborhood of $u$ in $V'$ and $\xi \in N^*_ {w}$.

The result follows by compactness of $M$ (using prop. \ref{transverse}).
\end{proof}

Here are the proofs of the previous two results:

\begin{proof}[Proof of \ref{parametrix:main}]
Consider the data $(U_{i},V_{i}), (U'_{i},V'_{i})$ and $P'_{i}$ of \ref{localization} associated to $P$. Let $(\chi_{i})_{i \in I}$ be a partition of unity associated to the cover $(V_{i})_{i \in I}$.  

Since the (plain) symbol of $P'_{i}$ is invertible over $\overline V_i$, there exist   $T_{i} \in \cP(U'_{i},V'_{i};\Omega^{1/2})$ of order $-m$ whose (plain) principal symbol is $\sigma_{m}(P'_{i})^{-1}$ in a neighborhood of $\overline V_i$. Then $P'_iT_i\chi_i-\chi\in \cP_c^{-1}(U'_{i},V'_{i};\Omega^{1/2})$ (we use $p^i_!$ to make this composition).

There is an operator $R_i\in \cP_c^k(U_{i},V_{i};\Omega^{1/2})$ whose image in $\Pseudodif(\cU,\cV)$ is $\chi_{i} S$ up to regularizing operators. Put then $Q=\sum_{i\in I} T_iR_i$.
\end{proof}

\begin{proof}[Proof of \ref{parametrix}]
Lemma \ref{parametrix:main} allows us to follow the classical proof: \\
First construct  $Q_{0}\in \cP_c^{-m}(\widetilde U,\widetilde V;\Omega^{1/2})$ such that $I-Q_0P$ is of negative order. By putting $Q_{k} = Q_{0}(I - PQ_{0})^{k}=(I - Q_{0}P)^{k}Q_0$ we obtain a sequence of operators of order $-m-k$, $i \in \N$. From \ref{smoothing} it follows that there exists $Q$ of order $-m$ (asymptotically the sum of the $Q_{k}$) such that $I-PQ$ and $I-QP$ are regularizing.
\end{proof}

\subsection{Square roots}

\begin{theorem}[square roots] \label{self-adjoint}
If $P \in \cP_c^{2m}(\widetilde U,\widetilde V;\Omega^{1/2})$ is self-adjoint of even order and $\tilde\sigma_{2m}(P)>0$, there is a self-adjoint $Q\in \cP_c^{m}(\widetilde U,\widetilde V;\Omega^{1/2})$ selfadjoint such that $P-Q^2$ is smoothing.
\end{theorem}
\begin{proof}
We use lemma \ref{localization} and the notation there. Coming back to its proof, we may assume that the (plain) symbol of $P'_i$ restricted to $V_i$ is $>0$. Let then $Q'_i\in  \cP_c^{-1}(U'_{m},V'_{i};\Omega^{1/2})$ that we may assume self-adjoint, such that the restriction to $V_i$ of its plain principal symbol is $\sqrt{\sigma_{2m}(P'_i)}$. Let $\chi_i^2$ be a partition of the identity adapted to $V_i$, and put $Q_0=\sum_{i\in I} \chi_i Q'_i\chi_i$. 
Then $P-Q_{0}^{2}$ is of order $2m-1$. Note that $Q_0$ is elliptic

Suppose we constructed $Q_{0},\ldots, Q_{n-1}$ self-adjoint, such that $Q_j$ has order $m-j$ and $R_{n} = P-(\sum_{j=0}^{n-1}Q_{j})^{2}$ is self adjoint of order $2m-n$. Thanks to lemma \ref{parametrix:main}, we find $Q_n$ of order $m-n$ such that $2Q_0Q_n-R_n$ has order $m-n-1$. It is a consequence of prop. \ref{products} that $Q_0$ and $Q_n$ commute up to lower order, therefore $2Q_nQ_0-R_n$ has order $m-n-1$. Since $Q_0$ and $R_n$ are selfajoint, we may replace $Q_n$ by $1/2(Q_n+Q_n^*)$. Hence $Q_{n}$ is a sequence that satisfies \ref{smoothing}. Put $Q'$ the asymptotic sum of the $Q_{j}$s and $Q = \frac{Q' + Q'^{*}}{2}$. This is self-adjoint and also an asymptotic sum for $Q_{n}$. By construction, $P-Q^2$ is smoothing.
\end{proof}


\section{The extension of zero order pseudodifferential operators}

Let us begin by a remark that will allow us to assume that the manifold $M$ is compact.

\begin{remark}\label{Mbecomescompact}
\begin{enumerate}
\item Let $M'$ be an open subset of $M$. Then $C_c^\infty (M')\cF$ is a foliation $\cF'$ on $M'$. A bi-submersion $(U,t,s)$ for $(M,\cF)$ restricts to a bi-submersion of $(M',\cF')$ by putting $U'=\{u\in U;\ s(u)\in M', \ t(u)\in M'\}$. In this way, an atlas $\cU$ of $(M,\cF)$ restricts to an atlas of $\cU'$ of  $(M',\cF')$. By extending compactly supported functions on $\cU'$, we embed $C^*(\cU')$ into  $C^*(\cU)$.

\item Assume $M'$ is relatively compact in $M$. Then there exists $f\in C_c^\infty(M)$ which is everywhere nonzero on $M'$. Note that $\{fX;\ X\in \cF\}$ is a foliation on $M$ which has the same restriction to $M'$ as $M$. There is a compact manifold $M''$ which contains an open subset diffeomorphic to a neighborhood of the support of $f$. Then $M''$ carries a foliation which has the same restriction to $M'$ as $\cF$.
\end{enumerate}
\end{remark}

\begin{lemma}
\label{lemmanegative} Every sufficiently negative order  pseudodifferential operator defines an element in $C^*(\cU)$. More precisely, given a bi-submersion $(U,t,s)$ and an identity bisection $V\subset U$, let $P\in \cP_c^{-m}(U,V;\Omega^{1/2})$ with $m$ strictly bigger than the dimension of the fibers of $s$ and $t$, then $\tilde \theta(P)\in C^*(\cU)$.
\end{lemma}

\begin{proof}
A continuous function with compact support in $U$ defines an element in $C^*(\cU)$ (thanks to the $L^1$ estimate and by density of $C_c^\infty(U)$). Now, if $a$ is of sufficiently negative order, the integral $$\int\!\! \int _{N^*v}a(v,\xi)\chi(u)e^{i\langle u,\xi\rangle}$$ makes sense and thus the distribution $P$ is actually a continuous function with compact support in $U$. 
\end{proof}

\begin{theorem} \label{thm:extension}
\begin{enumerate}
\item  Negative order pseudodifferetial operators are in $C^*(\cU)$, as well as those zero order operators whose principal symbol vanishes on $\cF^*$.
\item Zero order pseudodifferential operators define bounded multipliers of
the $C^*$-algebra of the foliation.
\end{enumerate}
\end{theorem}

\begin{proof} Using remark \ref{Mbecomescompact}, we may assume $M$ is compact. By \ref{self-adjoint}, if $\|\tilde \sigma_P^0\|<t$, there is $Q\in \Pseudodif(\cU,\cV)$ such that $P^*P+Q^*Q=t^2+R$ where $R$ is of negative enough order, so that it belongs to the $C^*$-algebra of the foliation (in fact it can even be taken smoothing).  We have $(Pf)^*(Pf)+(Qf)^*(Qf)=t^2\, f^*f+f^*Rf$ for all $f\in \cA(\cU)$.

It follows that:
 \begin{itemize}
\item  $\|Pf\|\le k\|f\|$, where $k=\sqrt{t^2+\|R\|}$, hence $P$ extends to a bounded multiplier and (b) follows.

\item if $\sigma_P^0=0$, then in the quotient $C^*$-algebra $\overline{\Pseudodif(\cU,\cV)}/C^*(\cU)$, the norm of $P$ is $\le t$ for all $t>0$, whence $P\in C^*(\cU)$.
\qedhere
\end{itemize}
 \end{proof}

We thus have an exact sequence of $C^*$-algebras 
\begin{eqnarray}\label{0-order}
0\to C^*(M,\cF)\to \Psi^*(M,\cF)\to B\to 0 
\end{eqnarray}
where $\Psi^*(M,\cF)$ denotes
the closure of the algebra of zero order pseudodifferential
operators with respect to multiplier norm and order $0$ symbol. The algebra $B$ is a quotient of the algebra $C_0(S^*\cF)$ of continuous functions on the cosphere ``bundle''. As discussed in remark \ref{remarksymbol}, if the groupoid $\cG_\cU$ is longitudinally smooth, then $B=C_0(S^*\cF)$.

\section{Longitudinally elliptic  operators of positive order} 
\label{sect2.8}

In this section we assume that $M$ is compact.

\subsection{Longitudinally elliptic operators and regular multipliers}

Recall \cite{Baaj, BaajJulg, Woronowicz} that an unbounded multiplier $T$ of a $C^*$-algebra is said to be \emph{regular} if it is densely defined, its adjoint is densely defined and its graph is \emph{orthocomplemented}, which means that $A\oplus A=G\oplus G^\perp$, where $G=\{(x,Tx);\ x\in \dom\, T\}$ is the graph of $T$ and $G^\perp=\{(T^*y,-y);\ y\in \dom\,T^*\}$ its orthogonal complement for the obvious $A$ valued scalar product in $A\oplus A$. 

Let $\Pi $ be a non degenerates representation of $A$. It extends to a representation $\tilde \Pi$ of the multiplier algebra $\cM(A)$. Every regular unbounded multiplier $T$ of $A$ gives rise to a closed operator $\widehat {\Pi}(T)$ whose graph is the closure of $\{(\Pi(a)\xi,\Pi(Ta)\xi);\ a\in \dom\, T;\ \xi\in H_\Pi\}$. The adjoint of $\widehat {\Pi}(T)$ is $\widehat {\Pi}(T^*)$. In particular, if $T$ is self adjoint, so is $\widehat {\Pi}(T)$.

\bigskip In \cite[\S 3,4]{Vassout} Vassout proved that elliptic pseudodifferential operators (of positive order) on a Lie groupoid $G$ give rise to regular operators. The proof in \cite{Vassout} can be adapted to our setting to show:

\begin{theorem}
If $P \in \Pseudo m(\cU,\cV)$ is the image of a longitudinally elliptic operator of order $m$, then $\overline{P}$ is a regular multiplier on $C^{*}(\cU)$ ($m>0$).
\end{theorem}

\begin{proof}
For every $S\in \Pseudo 0(\cU,\cV)$, denote by $\overline S$ its closure which is a multiplier of $C^*(\cU)$.

Let $Q$ be a parametrix of $P$ and write $I-QP=R$. Let $T\in \cL(C^*(\cU)\oplus C^*(\cU))$ be the (adjointable) operator of  the $C^*$-module $C^*(\cU)\oplus C^*(\cU)$ with matrix $\begin{pmatrix}
\overline R&\overline{Q}\\
\overline{PR}&\overline{PQ}
\end{pmatrix}.$ 

The restriction to $\theta(\cA)\oplus \theta(\cA)$ of $T^2$ has matrix $\begin{pmatrix}
(R+QP)R& (R+QP)Q\\
P(R+QP)R&P(R+QP)Q
\end{pmatrix}$, whence $T^2=T$ (as $R+QP=I$) .

\begin{itemize}
\item  Since $T$ is an (adjointable) idempotent element in $ \cL(C^*(\cU)\oplus C^*(\cU))$, its range is orthocomplemented.
\item  Since $T$ is continuous, and $\theta(\cA(\cU))\oplus \theta(\cA(\cU))$ is dense in $C^*(\cU)\oplus C^*(\cU)$, we deduce that the range of $T$ is the closure of $T(\theta(\cA(\cU))\oplus \theta(\cA(\cU)))$
\item If $(x,y)\in \theta(\cA(\cU))$, we find $T(x,y)=(Rx+Qy,P(Rx+Qy))$. If furthermore $y=Px$, $T(x,y)=(x,y)$. It follows that $T(\theta(\cA(\cU))\oplus \theta(\cA(\cU)))$ is the graph of $P$. \end{itemize}

We just proved that the closure of the graph of $P$ is orthocomplemented, \ie $\overline P$ is regular.
\end{proof}

\begin{remarks}
In the same way we may adapt the proofs of \cite{Vassout} to our setting to prove:
\begin{enumerate}
\item Any two longitudinally elliptic operators $P, P' \in \Pseudodif(\cU,\cV)$ of the same order have the same domain.
\item Longitudinally elliptic operators define a filtration of $C^{*}(\cU)$ by Sobolev modules. If $P$ is of order $k > 0$ then
\begin{itemize}
\item $H^{k}(P) = {\rm dom}\, P$ with scalar product $\langle \alpha,\beta \rangle_{k} = \langle P\alpha,P\beta \rangle + \langle \alpha,\beta \rangle$
\item $H^{-k}(P)$ is the completion of $C^{*}(\cU)$ with the norm $\|\xi\|_{-k} = \|(1+P^{*}P)^{-1/2}\xi\|$.
\end{itemize}
This filtration satisfies the following properties:
\begin{itemize}
\item If $k > k'$ then the identity on $\cA_{\cU}$ extends to a compact morphism of Hilbert modules $i_{k,k'} : H^{k} \hookrightarrow H^{k'}$.
\item Any $P \in \Pseudodif(\cU,\cV)$ of order $m$ defines an element of $\cL(H^{k};H^{k-m})$ for any $k$.
\end{itemize}
\item Using the Sobolev spaces above one can define an algebra $\Psi^{-\infty}(\cU)$ of smoothing pseudodifferential operators without compact support by calling an operator $R$ smoothing iff $R \in \bigcap_{s,t \in \R}\cL(H^{s};H^{t})$. It follows that $\cA_{\cU} \subset \Psi^{-\infty}(\cU)$ is a dense subalgebra.
\end{enumerate}
\end{remarks}

\subsection{Application: Laplacian of a singular foliation}

As a particular case we may construct a laplacian operator for every foliation and prove that it is a positive self-adjoint operator of $L^{2}(M)$.

Let $(M,\cF)$ be a foliation. Every vector field $X\in \cF$ defines a differential hence pseudodifferential operator $X\in \Pseudodif(\cU,\cV)$ (\cf example \ref{vectorfields}).

Since $M$ is compact, $\cF$ is generated by finitely many vector fields $X_{1},\ldots, X_{N}$. 

\begin{definition}
The element $\Delta=\sum_{k=1}^N X_k^*X_k$ is called a {\em Laplacian} of the foliation $\cF$.
\end{definition}

From the definition of $\Delta$ we have:

\begin{theorem} \label{laplace}
A Laplacian $\Delta$ is a formally self adjoint elliptic operator of order $2$; it therefore defines a regular (unbounded) self adjoint multiplier of $C^{*}(\cU)$.\hfill$\square$
\end{theorem}

We may apply that to the natural representation of $C^*(\cU)$ on $L^2(M)$ (which can be seen as the integration of the trivial representation $(\lambda,\C)$ of the groupoid in the sense of \cite[section 5.1]{AndrSk} where $\lambda $ is the Lebesgue measure on $M$).

\begin{corollary}
The Laplacian $\Delta$ defines an unbounded, self-adjoint operator $\Delta$ of $L^{2}(M)$. In other words, take vector fields $X_1,\ldots,X_N$ on a compact manifold. Assume that the module they generate is a foliation, \ie for every $(i,j)$ there exist $f_{i,j,k}\in C^\infty(M)$ such that $[X_i,X_j]=\sum_kf_{i,j,k}X_k$. Then the closure in $L^2(M)$ of $\sum_{k=1}^NX_j^*X_j$ is self adjoint.
\end{corollary}

\begin{remarks}
\begin{enumerate}
\item One can apply theorem \ref{laplace} to other natural representations of $C^*(\cU)$. One may for instance take the representation on $L^2$ of a leaf. There, the Laplacian is elliptic and the difficulty comes from the fact that the leaf may not be compact.
\item The spectrum of the image of every regular operator - and in particular of $\Delta$ - is the same if we take two weakly equivalent representations of  $C^*(\cU)$. This should apply  if we compare the representation in $L^2(M)$ and a representation in $L^2$ of a dense leaf.
\end{enumerate}
\end{remarks}

\printindex

\end{document}